\documentclass[12pt,a4paper]{article}

\usepackage{amsmath,amsthm}

\newtheorem{theo}{Theorem}[section]
\newtheorem{corol}[theo]{Corollary}
\newtheorem{prop}[theo]{Proposition}
\newtheorem{lem}[theo]{Lemma}

\theoremstyle{definition}
\newtheorem{defi}[theo]{Definition}
\newtheorem{example}{Example}

\newenvironment{rem}{\bigskip\emph{Remark.}}{}

\newenvironment{nota}{\emph{Notation.}}{}

\def\Rr{\mathbf{R}}

\def\Zz{\mathbf{Z}}

\let\bord\partial
\let\bydef\emph
\let\epsi\varepsilon

\def\var{mani\-fold}
\def\svar{sub\-mani\-fold}
\def\irr{ir\-redu\-ci\-ble}
\def\incomp{in\-com\-pres\-sible}

\def\Ric{\mathrm{Ric}}
\def\Rm{\mathrm{Rm}}
\def\vol{\mathrm{vol}}
\def\hyp{\mathrm{hyp}}

\def\Rmin{R_\mathrm{min}}
\def\Rmax{R_\mathrm{max}}

\begin{document}

\title{Some applications of Ricci flow to 3-manifolds}
\author{Sylvain Maillot}

\maketitle

\section{Introduction}

The purpose of this text is to describe some applications of Ricci flow to questions about the
topology and geometry
of $3$-\var s. The most spectacular achievement in the recent years was the proof of
W.~Thurston's Geometrization Conjecture by G.~Perelman~\cite{Per1,Per2,Per3}.  In Sections~2--6
we outline a proof of the Geometrization Conjecture which is based in part on Perelman's
ideas. This is joint work with L.~Bessi\`eres, G.~Besson, M.~Boileau, and J.~Porti. Some details
have already appeared in the preprint~\cite{b3mp:hyper6}; the rest will be contained
in our forthcoming monograph~\cite{b3mp:book}. In Section~7 we announce some results
on the topological classification of (possibly noncompact) $3$-manifolds with positive
scalar curvature. These results follow from an extension of those discussed in Section~6, and
are joint work with L.~Bessi\`eres and G.~Besson. Details will appear elsewhere.

The theory of low-dimensional Ricci flows relies on techniques and insights from riemannian geometry,
geometric analysis, and topology. In this text, we deliberately adopt the topologist's viewpoint; thus we shall focus on the topological and geometric arguments rather than the analytic aspects.
The Ricci flow is considered here as a tool to prove topological and geometric theorems.
Therefore, we sometimes do not prove the strongest possible estimates on Ricci flow solutions, but
rather strengthen the theorems that are used to deduce topological consequences
(e.g.~Theorem~\ref{reconnait topo} or Proposition~\ref{summary}) by weakening their
hypotheses. It should go
without saying that the Ricci flow as a mathematical object is worth studying for its
own sake, which justifies trying to obtain the best possible results on the analytic side.

This text is mostly intended for nonexperts. The prerequisites are basic concepts in
algebraic topology and differential geometry; previous knowledge of $3$-\var\ theory
or Ricci flow is of course useful, but not necessary. We have endeavored to
make the various  parts of the proof as independent from each other as possible,
in order to clarify its overall structure.

Special paragraphs marked with an asterisk (*) are geared toward the experts who wish to 
understand the differences between the proof of the Geometrization Theorem presented
here and Perelman's proof. Only in those paragraphs is some familiarity with Perelman's papers implicitly
assumed.

\paragraph{Acknowledgments}
I would like to thank all the people who have, directly or indirectly, helped me understand
Perelman's Ricci flow papers. In addition to my above-mentioned collaborators, this includes the
participants of the workshops in Grenoble 2004 and M\"unchen 2005, especially Bernhard
Leeb, Thomas Schick, Hartmut Weiss and Jonathan Dinkelbach. I am also indebted to
Bruce Kleiner and John Lott for writing their notes~\cite{Kle-Lot}, as well as the Clay
Mathematics Institute for financial support.

Numerous conversations, in particular with Thomas Delzant and Olivier Biquard, have helped
shape my thoughts on the subject. I also thank the organizers of the seminars and
conferences where I have had the opportunity to present this work. The present text
is based in part on notes from those talks. Lastly, I would like to thank Laurent Bessi\`eres and G\'erard Besson for their
remarks on a preliminary version of this paper.

\section{Geometrization of 3-manifolds}\label{sec:geom}

\emph{All $3$-manifolds considered in this text are assumed to be smooth, connected, and orientable.} The $n$-dimensional sphere (resp.~$n$-dimensional
real projective space, resp.~$n$-dimensional torus) is denoted by $S^n$ (resp.~$RP^n$,
resp.~$T^n$).
General references for this section are~\cite{scott:geometries,boileau:bourbaki,bmp,hempel,jaco:lect}.

We are mostly interested in compact manifolds, and more precisely \bydef{closed}
manifolds, i.e., compact manifolds whose boundary is empty.
A closed $3$-manifold is \bydef{spherical} if it admits a riemannian metric of constant
positive sectional curvature. Equivalently, a $3$-manifold is spherical if it can be obtained as
a quotient of $S^3$ by a finite group of isometries acting freely. In particular, the universal
cover of a spherical $3$-manifold is diffeomorphic to $S^3$.

In this text, a $3$-manifold $H$ with empty boundary is called \bydef{hyperbolic} if it admits a complete riemannian metric of constant sectional curvature $-1$ and finite volume. Such a
metric is called a \bydef{hyperbolic metric}. A hyperbolic manifold may be closed, or have finitely
many ends, called \bydef{cusps}, which admit neighborhoods diffeomorphic to $T^2\times [0,+\infty)$. The hyperbolic metric, which by Mostow rigidity is unique up to isometry, is denoted by $g_\hyp$. By extension, a compact
manifold $M$ with nonempty boundary is called hyperbolic if its interior is hyperbolic. Thus
in this case $\bord M$ is a union of tori.

A $3$-manifold is \bydef{Seifert fibered}, or simply \bydef{Seifert}, if it is the total space
of a fiber bundle over a $2$-dimensional orbifold. It is well-known that all spherical $3$-manifolds are
Seifert fibered.  Seifert manifolds have been classified, and form a
well-understood class of $3$-\var s. Among them, spherical \var s are exactly those
with finite fundamental group.

To state the Geometrization Theorem, we still need two definitions: a $3$-manifold $M$
is \bydef{irreducible} if every embedding of the $2$-sphere into $M$ can be extended to
an embedding of the $3$-ball. An orientable embedded surface $F\subset M$ of positive genus is called
\bydef{incompressible} if the group homomorphism $\pi_1F\to \pi_1 M$ induced by
the inclusion map is injective.

The main purpose of this article is to present a proof of the following result:
\begin{theo}[Perelman]\label{geometrisation}
Let $M$ be a closed, irreducible $3$-manifold. Then $M$ is hyperbolic, Seifert fibered, or
contains an embedded incompressible torus.
\end{theo}

Theorem~\ref{geometrisation}, together with the Kneser-Milnor prime
de\-com\-po\-si\-tion theo\-rem (see e.g.~\cite{hempel},) the torus splitting theorem
of Jaco-Shalen~\cite{js:seifert} and Johannson~\cite{joh:hom}, and
Thurston's hyperbolization theorem for Haken manifolds~\cite{otal:fibered,otal:haken,kap:hyp},
implies Thurston's original formulation of the Geo\-me\-tri\-za\-tion
Conjecture in terms
of a canonical decomposition of any compact $3$-\var\ along spheres and tori into
`geometric' $3$-\var s admitting locally homogeneous riemannian
metrics.\footnote{Of the famous `eight geometries' that uniformize those locally
homogeneous $3$-manifolds,
six correspond to Seifert manifolds (including spherical geometry), one is hyperbolic
geometry, and the last one, $\mathbf{Sol}$, does not appear in the above statement.
(From our viewpoint, $\mathbf{Sol}$ \var s are just special $3$-\var s containing \incomp\ tori.)}

Theorem~\ref{geometrisation} implies the Poincar\'e Conjecture: indeed, by Kneser's theorem and
van Kampen's theorem,
it is enough to prove this conjecture for \irr\ $3$-\var s. Now if $\pi_1M$ is trivial, then $M$
cannot be hyperbolic or contain an incompressible torus. Hence $M$ is Seifert. As mentioned earlier, Seifert manifolds with finite fundamental group are spherical. It follows that $M$ is
a quotient of the $3$-sphere by a trivial group, i.e.~the $3$-sphere itself.

By a straightforward extension of the above argument, one can deduce
from Theorem~\ref{geometrisation} the following strengthening of the Poincar\'e Conjecture: 

\bigskip
{\bf Elliptization Conjecture} Every closed $3$-\var\ with finite fundamental
group is spherical. 
\bigskip

When $\pi_1M$ is infinite, it was known from~\cite{scott:fake,mess:seifert,tukia:hom,gabai:conv,
cj:conv} (see also~\cite{maillot:center} and~\cite[Chapter 5]{bmp})
that if $\pi_1 M$
has a subgroup isomorphic to $\Zz^2$, then it is Seifert fibered or contains an embedded
\incomp\ torus. Thus, prior to Perelman's work,  the only remaining open question was the following:

\bigskip
{\bf Hyperbolization Conjecture} If $M$ is a closed, irreducible $3$-\var\ whose fundamental
group is infinite and does not contain a subgroup isomorphic to $\Zz^2$, then $M$ is hyperbolic.
\bigskip

This statement is also a direct consequence of Theorem~\ref{geometrisation}.

The dichotomy finite vs infinite fundamental group will not appear directly in the proof
presented here, but rather via properties of the higher homotopy groups.
The connection is given by  the following well-known lemma:
\begin{lem}\label{lem:classic}
Let $M$ be a closed, \irr\ $3$-\var. Then the following are equivalent:
\begin{enumerate}
\item $\pi_1M$ is infinite;
\item $\pi_3M$ is trivial;
\item $M$ is \bydef{aspherical}, i.e.~$\pi_k M$ is trivial for all $k\ge 2$.
\end{enumerate}
\end{lem}

\begin{proof}
It follows from irreducibility of $M$ and the Sphere Theorem that $\pi_2M$ is trivial.
Let $\tilde M$ be the universal cover of $M$. Then $\tilde M$ is $2$-connected.
By the Hurewicz Isomorphism theorem, we have $\pi_3\tilde M\cong H_3 \tilde M$.
Hence $\pi_3 M$, which is isomorphic to $\pi_3 \tilde M$, vanishes if and only if $\tilde M$
is noncompact. This proves that (i) and (ii) are equivalent.

It is immediate that (iii) implies (ii). For the converse, apply the Hurewicz theorem inductively
to $\tilde M$.
\end{proof}

\section{The Ricci Flow approach}\label{sec:ricci}

\begin{nota}
If $g$ is a riemannian metric, we denote by $\Rmin(g)$
the minimum of its scalar curvature,
by $\Ric_g$ its Ricci tensor, and by $\vol(g)$ its volume.
\end{nota}

\medskip
Let $M$ be a closed, irreducible $3$-manifold. In the Ricci flow approach to
geometrization, one studies solutions of the evolution equation
\begin{equation}\label{rf}
\frac{dg}{dt} = -2 \Ric_{g(t)},
\end{equation}
called the \bydef{Ricci flow equation}, which was introduced by R.~Hamilton. A solution
is an \bydef{evolving metric} $\{g(t)\}_{t\in I}$, i.e.~a $1$-parameter family of riemannian metrics on $M$
defined on an interval $I\subset\Rr$.
In~\cite{hamilton:three}, Hamilton proved that for any metric $g_0$ on $M$, there exists $\epsi>0$ such that Equation~(\ref{rf}) has a unique solution defined
on $[0,\epsi)$ with initial condition $g(0)=g_0$. Thus there exists $T\in (0,+\infty]$ such
that  $[0,T)$ is the maximal interval where
the solution to~(\ref{rf}) with initial condition $g_0$ is defined. When $T$ is finite,
one says that Ricci flow has a \bydef{singularity} at time $T$. Ideally, one would like to
see the geometry of $M$ appear by looking at the metric $g(t)$ when $t$ tends to $T$
(whether $T$ be finite or infinite.)
To understand how this works, we first consider some (very) simple examples, where
the initial metric is locally homogeneous.

\begin{example}
If $g_0$ has constant sectional curvature $K$, then the solution is given by
$g(t)=(1-4Kt) g_0$. Thus in the spherical case, where $K>0$, we have $T<\infty$, and as $t$
goes to $T$, the manifold shrinks to a point while remaining of constant positive curvature.

By contrast, in the hyperbolic case, where $K<0$, we have $T=\infty$ and $g(t)$ expands
indefinitely, while remaining of constant negative curvature. In this case,  the \emph{hyperbolically rescaled} solution $\tilde g(t):= (4t)^{-1} g(t)$ converges to the
metric of constant sectional curvature $-1$.
\end{example}
 
\begin{example}
If $M$ is the product of a circle with a surface of genus at least $2$, and $g_0$ is
a product metric whose second factor has constant curvature, then $T=\infty$; moreover, the
metric is constant on the first factor, and expanding on the second factor.

In this case, $g(t)$ does not have any convergent rescaling. However, one can
observe that the hyperbolically rescaled solution $\tilde g(t)$ defined above \bydef{collapses with bounded curvature}, i.e.~has bounded curvature and injectivity radius
going to $0$ everywhere as $t$ goes to $+\infty$. Manifolds admitting collapsing sequences
of riemannian metrics under various curvature bounds have been studied by several authors,
starting with the seminal work of Cheeger-Gromov~\cite{Che-Gro1,Che-Gro2} (see also the surveys~\cite{pansu:effondrement,fukaya:collapsing}, as well as the recent paper~\cite{sy} and the references
therein.)
Thus one is led to expect that these techniques can be used to deduce topological
information on $M$ from the large-scale behavior of $\tilde g(t)$ in this case.

One can give similar formulae for other locally homogeneous Ricci flows (see e.g.~\cite{lott:long}.)
\end{example}

An (overly) optimistic program based on the previous examples would go as follows:
if $\pi_1M$ is finite, prove that any Ricci flow solution has a singularity at some finite time $T$, and
that the spherical metric can be obtained as a limit of rescalings of $g(t)$ when $t$ goes to $T$.
If $\pi_1M$ is infinite, show that Ricci flow is defined for all times, and study the long time
behavior of the
hyperbolically rescaled solution $\tilde g(t):= (4t)^{-1} g(t)$ in order to recognize the
topological type of $M$; for instance, if $M$ has a hyperbolic metric, then this metric ought to
appear as the limit of $\tilde g(t)$ as $t$ tends to infinity.

Several important results were obtained by
Hamilton~\cite{hamilton:three,hamilton:four,hamilton:nonsingular} in this direction,
among which we state just two. When $g_0$ has positive Ricci curvature,
then $T<\infty$, and the volume-rescaled Ricci flow $\vol(g(t))^{-2/3} g(t)$ converges to a round metric
as $t\to T$.  If 
$T=\infty$ and $\tilde g(t)$
has uniformly bounded sectional curvature, then it  converges or collapses, or $M$
contains an incompressible torus.\footnote{Hamilton's original results were formulated in terms
of normalized Ricci flow. We have restated them so that they fit better in our discussion.}

The general case, however, is more difficult, because it sometimes happens that $T<\infty$ while
the behavior of $g(t)$ as $t$ tends to $T$ does not allow to determine the
topology of $M$. One possibility  is the so-called \bydef{neck pinch},
where part of $M$ looks like a thinner and thinner cylindrical neck as one approaches
the singularity. This can happen
even if $M$ is irreducible (see~\cite{ak:neck} for an example where $M=S^3$); thus
neck pinches may not give any useful information on the topology of $M$.
See~\cite{hamilton:singular} and~\cite[Chapters 8 and 9]{Cho-Kno} for a discussion of what was known on formation of singularities
prior to Perelman's work.

A solution to this problem was found by G.~Perelman, inspired by ideas of
Hamilton~\cite{hamilton:isotropic}. In~\cite{Per2},  Perelman
explains how to construct a kind of  generalized solution to the Ricci flow equation, which he
calls \emph{Ricci flow with $\delta$-cutoff}. This type of solution exists for all time unless it leads to
a metric that is sufficiently controlled so as to allow one to recognize the topology of the manifold.
Several slightly different ways to make this construction precise are given in~\cite{Kle-Lot,Mor-Tia,Cao-Zhu}. We shall work with closely related objects which we call \emph{weak solutions}
of the Ricci flow equation.

The first part of the proof of Theorem~\ref{geometrisation} is to prove an
existence theorem for weak solutions. The proof then splits into two cases.

If $\pi_1M$ is finite, then by Lemma~\ref{lem:classic}, we have $\pi_3M\neq 0$. Following~\cite{cm:question},
we use this fact
to associate to any riemannian metric $g$ a quantity called its \emph{width}  $W(g)$. This
invariant can be studied via minimal surface theory. In particular, one can control the way it
varies with time in a weak solution $\{g(t)\}$. This
is used to prove that for some finite time $t_0$, the metric $g(t_0)$ belongs to a special
class of metrics, called \emph{locally canonical} metrics. This permits to recognize the
topology of $M$.

If $\pi_1M$ is infinite, then we need to refine the existence theorem for weak solutions.
We prove that such solutions exist for all time, and enjoy additional properties.
Then we study  the long time behavior of the corresponding hyperbolically rescaled solutions to deduce
topological consequences.

In Section~\ref{sec:poincare}, we give a precise definition of weak solutions, state
an existence result, and explain how to deduce the Elliptization Conjecture. The case
where $\pi_1M$ is infinite is tackled in Section~\ref{sec:hyper}. In Section~\ref{sec:solution},
we discuss in more detail the construction of weak solutions.

\section{Canonical neighborhoods, weak solutions and elliptization}
\label{sec:poincare}

\subsection{Canonical neighborhoods}
Let $\epsi$ be a positive number. The \bydef{standard} $\epsi$-\bydef{neck}
is the riemannian product $N_\epsi:=S^2\times (-\epsi^{-1},\epsi^{-1})$,
where the $S^2$ factor is round of scalar curvature $1$. Its metric is denoted by $g^\epsi_0$.

\begin{defi}\label{def:cn}
Let $\epsi$ be a positive number, 
$(M,g)$ be a riemannian $3$-\var\ and $x$ be a point of $M$. A neighborhood $U\subset M$ of
$x$ is called a \bydef{weak} $\epsi$-\bydef{canonical neighborhood} if one of the
following holds:
\begin{enumerate}
\item There exist a number $\lambda>0$ and a
$\mathcal C^{[\epsi^{-1}]+1}$-diffeomorphism $f:
(N_\epsi,*) \to (U,x)$ (where $N_\epsi$ is the standard $\epsi$-neck, and $*$ is a
basepoint in $S^2\times \{0\}$) such that the $\mathcal C^{[\epsi^{-1}]+1}$-norm of the tensor
$\phi^* (\lambda g) - g^\epsi_0$ is less than $\epsi$ everywhere. 
In this case, we say that $U$ is an $\epsi$-\bydef{neck} centered at $x$.
\item $U$ is the union of two sets $V,W$ such that $x\in V$, $V$ is
a closed topological $3$-ball,  $\bar W \cap V=\bord V$, and $W$ is an $\epsi$-neck. Such
a $U$ is called an $\epsi$-\bydef{cap} centered at $x$.
\item $U$ is a spherical $3$-\var\ (hence $U=M$ since $M$ is connected).
\end{enumerate}
\end{defi}

This notion is useful for topological purposes because of the following result (cf.~\cite[Proposition A.25]{Mor-Tia}):

\begin{theo}\label{reconnait topo}
Let $(M,g)$ be a closed riemannian $3$-manifold. If every point of $(M,g)$ has a weak
${10}^{-2}$-canonical neighborhood, then $M$ is spherical or diffeomorphic to
$S^2\times S^1$.
\end{theo}

\begin{proof}[Sketch of proof]
If some point has
a neighborhood of the third type, there is nothing to prove. Otherwise, there are two cases: if
all points are centers of necks, then by piecing together those necks, one obtains a
fibration of $M$ by $2$-spheres. Since $M$ is orientable, it must be diffeomorphic
to $S^2\times S^1$. If some point is the center of a cap, then one shows that $M$ is
either the union of two caps, or the union of two caps and a chain of necks connecting
them. In either case, $M$ is diffeomorphic to $S^3$.
\end{proof}

Theorem~\ref{reconnait topo} motivates the following definition:
\begin{defi}
A riemannian metric $g$ on a $3$-\var\ $M$ is \bydef{locally canonical} if
every point of $(M,g)$ has a weak ${10}^{-2}$-canonical neighborhood.
\end{defi}

\begin{rem}(*)
Definition~\ref{def:cn} is weaker than Perelman's in several respects. Perelman's definition includes additional
geometric information such as estimates on the curvature or its derivatives, which are not needed
for Theorem~\ref{reconnait topo}. He also considers caps  that are diffeomorphic to the punctured real projective
$3$-space instead of the $3$-ball. We do not need to do this because we shall soon restrict
attention to $3$-\var s without embedding projective planes.

Case (iii) of Definition~\ref{def:cn} may appear somewhat artificial. A more natural definition would include some
geometric information related to the formation of singularities of the Ricci flow. However, this requires
splitting case~(iii) into two subcases, which creates complications irrelevant to the topological applications.
\end{rem}

\subsection{Weak solutions}
Until the end of this section, we suppose that $M$ is closed and  \irr. We also assume that $M$ is
$RP^2$-\bydef{free}, i.e.~does not contain
any \svar\ diffeomorphic to $RP^2$. This is not much of  a restriction because the
only closed,  \irr\ $3$-manifold that does contain an embedded copy of $RP^2$ is $RP^3$, which is
a spherical \var.

The next goal is to give the formal definition of weak solutions and state an existence
theorem. To motivate this, we first give a very broad outline of the construction, which
will be developped in Section~\ref{sec:solution}.

As we already explained, one of the main difficulties in the Ricci flow approach to geometrization
is that singularities unrelated to the topology of $M$ may appear. Using maximum principle
arguments, one shows that singularities in a $3$-dimensional Ricci flow can only occur when the scalar curvature tends to $+\infty$ somewhere. One of Perelman's major breakthroughs was to
give a precise local description of the geometry at points of large scalar curvature. In particular, he
showed that those points have weak canonical neighborhoods.

To solve the problem of singularities, we fix a large number $\Theta$, which plays
the role of a curvature threshold. As long as the maximum of the scalar
curvature is less than $\Theta$, Ricci flow is defined. If it reaches $\Theta$ at some time
$t_0$, then there are two
possibilities: if the minimum of the scalar curvature of the time-$t_0$ metric is large enough, then we shall see that this
metric is locally canonical, so that Theorem~\ref{reconnait topo} enables us to recognize
the topology of $M$.

Otherwise, we modify $g_0$  so that the maximum of the
scalar curvature of the new metric, denoted by $g_+(t_0)$, is at most $\Theta/2$. This modification is called \bydef{metric surgery}.
It consists in replacing some $\epsi$-caps by a special type of $\epsi$-caps called
\emph{almost standard caps}. Then we start the Ricci flow again, using $g_+(t_0)$ as new initial metric.
This procedure is repeated as many times as necessary. The main difficulty is to choose
$\Theta$ and do metric surgery in such a way that the construction can indeed by iterated.

We now come to the formal definitions. Let $M$ be a $3$-\var.

An \bydef{evolving metric} on
$M$ defined on an interval $[a,b]$ is a map $t\mapsto g(t)$ from $[a,b]$ to the
space of smooth riemannian metrics on $M$. It is \bydef{piecewise $\mathcal C^1$} if
there is a finite subdivision $a=t_0<t_1<\cdots < t_p=b$ such that for each $0\le i<p$,
there exists a metric $g_+(t_i)$ on $M$ such that the map $t\mapsto \bar g(t)$ defined
on $[t_i,t_{i+1}]$ by setting $\bar g(t_i):= g_+(t_i)$ and $\bar g(t):=g(t)$ if $t>t_i$ is
$\mathcal C^1$-smooth.

We often view an evolving metric $g(\cdot)$ as a $1$-parameter family of metrics indexed by
the interval $[a,b]$; thus we use the notation $\{g(t)\}_{t\in [a,b]}$.
For $t\in [a,b]$, we say that $t$ is \bydef{regular} if $g(\cdot)$ is $\mathcal C^1$-smooth
in a neighborhood of $t$. Otherwise $t$ is called \bydef{singular}. By definition, the set
of singular times is finite. If $t\in (a,b)$ is a singular time, then it follows from the
definition that the map $g(\cdot)$ is continuous from the left at $t$, and has a  limit from the right,
denoted by $g_+(t)$.

There are similar definitions where the domain of definition $[a,b]$ is replaced by an open or
a  half-open interval $I$. When $I$ has infinite length, the set of singular times is a discrete subset
of the real line, but it may be infinite.

\begin{defi}\label{def:ws}
Let $I\subset \Rr$ be an interval. A piecewise $\mathcal C^1$ evolving metric $t\mapsto g(t)$
on $M$ defined on $I$ is said to be a \bydef{weak solution} of the Ricci flow equation~(\ref{rf})
if
\begin{enumerate}
\item Equation~(\ref{rf}) is satisfied at all regular times.
\item For every singular time $t\in I$ one has
     \begin{enumerate}
        \item $\Rmin(g_+(t)) \ge \Rmin(g(t))$, and
      \item $g_+(t) \le g(t)$.
     \end{enumerate}
\end{enumerate}
\end{defi}

\medskip
We now state the main technical result on weak solutions needed to prove the Elliptization Conjecture
(cf.~\cite[Proposition 5.1]{Per2}):
\begin{theo}[Existence of weak solutions]\label{existence 1}
Let $M$ be a closed, \irr, $RP^2$-free $3$-manifold.
For every $T>0$ and every riemannian metric $g_0$ on $M$, there exists $T'\in (0,T]$ and
a weak solution $\{g(t)\}$ on $M$, defined on $[0,T']$, with initial condition $g(0)=g_0$, and such
that either
\begin{enumerate}
\item $T'=T$, or
\item $T'<T$ and $g(T')$ is locally canonical.
\end{enumerate}
\end{theo}

\bigskip

\begin{rem}
The statement of Theorem~\ref{existence 1} is slightly weaker than one might expect, since in case~(i) it
is not claimed that the solution exists for all time. A stronger statement is in fact true (cf.~Theorem~\ref{existence 2} below), but its proof is more involved.
\end{rem}

\begin{rem}(*)
There are a few differences between weak solutions and Perelman's Ricci flow with $\delta$-cutoff.
One obvious difference is that a weak solution is an evolving metric on a fixed manifold
rather than an evolving manifold. This simplification is made possible by the extra topological
assumptions on $M$.

Another, perhaps more significant difference is that surgery occurs \emph{before} the Ricci flow
becomes singular, rather than \emph{at} the singular time. Our construction is in this respect closer
in spirit to the surgery process envisioned by Hamilton~\cite{hamilton:isotropic}.
This point is discussed in more detail in Section~\ref{sec:solution}.
\end{rem}

\subsection{Proof of elliptization}

To deduce the elliptization part of Theorem~\ref{geometrisation}, all we need is the following result (cf.~\cite{Per3,cm:question}:)
\begin{theo}\label{extinction}
Let $M$ be a closed, \irr, $RP^2$-free $3$-manifold with finite fundamental group.
For each riemannian metric $g_0$ on
$M$ there is a number $T(g_0)>0$ such that if $\{g(t)\}$ on $M$ is a
weak solution defined on some interval $[0,T]$ and with initial condition $g(0)=g_0$,
then $T<T(g_0)$.
\end{theo}

Indeed, assume that Theorems~\ref{reconnait topo},~\ref{existence 1}, and~\ref{extinction} have
been proved.
Let $M$ be a closed, \irr, $RP^2$-free $3$-manifold such that $\pi_1M$ is finite, and let $g_0$ be
an arbitrary metric on $M$. Theorem~\ref{extinction} provides a positive
number $T(g_0)>0$ such that no weak solution with initial condition $g_0$ can exist
up to time $T(g_0)$. Let $T'>0$ and $\{g(t)\}_{0\le t< T'}$ be given by Theorem~\ref{existence 1} applied with $T=T(g_0)$.
Then Case~(ii) of the conclusion of that theorem must hold. This implies that $M$
has a locally canonical metric. Applying Theorem~\ref{reconnait topo} and noting
that $S^2\times S^1$ is not \irr, we conclude that $M$ is spherical.

Let us sketch the proof of Theorem~\ref{extinction}. Let $M$ be a $3$-\var\ satisfying
the hypotheses and $g_0$ be a riemannian metric on $M$. First we show,
following~\cite{cm:question}, that there is a constant $T$ such that there exists no Ricci flow solution
on $M$ defined on $[0,T]$, with initial condition $g(0)=g_0$.
It will then be straightforward to extend the argument to prove that there can in fact
exist no weak solution defined on the same interval.

Let $\Omega$ be the space of smooth maps $f:S^2\times [0,1] \to M$ such that $f(S^2\times
\{0\})$ and  $f(S^2\times \{1\})$ are points. Since $\pi_1M$ is finite, Lemma~\ref{lem:classic}
implies that $\pi_3M\neq 0$. It follows that there exists $f_0\in\Omega$ which is not
homotopic to a constant map. Let $\xi$ be the homotopy class of $f_0$.
We set $$W(g) := \inf_{f\in \xi} \max_{s\in [0,1]} E(f(\cdot,s)),$$
where $E$ denotes the energy $$E(f(\cdot,s)) = \frac12\int_{S^2} \vert \nabla_x f(x,s)\vert ^2\, dx.$$

Let $\{g(t)\}_{t\in [0,T]}$ be a Ricci flow solution such that $g(0)=g_0$. The function
$t\mapsto W(g(t))$ is continuous. Colding and Minicozzi~\cite{cm:question,cm:width} prove that there exists a constant $C\ge 0$ depending only on $g_0$ such that:

\begin{equation}\label{eq:W}
\frac{d}{dt} W(g(t)) \le -4\pi + \frac{3}{4(t+C)} W(g(t)),
\end{equation}
where $\frac{d}{dt} W(g(t))$ is to be interpreted as $$\limsup_{h>0, h\to 0} \frac{W(g(t+h)) - W(g(t))}
{h}$$ in case $W(g(t))$ is not differentiable at $t$ in the traditional sense.

Since the right-hand side is not integrable and $W(g(t))$ cannot become negative, this easily implies an upper bound on $T$ depending only on $C$, hence on $g_0$.

If we replace $\{g(t)\}_{t\in [0,T]}$ by a weak solution, then $W(g(t))$ need not be continuous
in $t$. However, condition~(ii)(b) in the definition of a weak solution immediately implies
that
\begin{equation}\label{eq:semi}
W(g_+(t))\le  W(g(t))
\end{equation}
if $t$ is a singular time. One can show that the inequality~(\ref{eq:W}) still holds at regular times. This is enough to conclude that the previous a priori upper bound on the existence
time of Ricci flow also applies to weak solutions.

\begin{rem}(*)
In the above argument, property (ii)(b) of the definition of weak solutions, which amounts to the fact
that surgery does not increase distances, was used in a crucial way. This part of the proof
has been simplified: indeed, when one uses Perelman's surgery construction, the pre-surgery metric
is not defined on the whole manifold, which makes the proof of~(\ref{eq:semi}) more complicated
(cf.~\cite{besson:bourbaki}).
\end{rem}

\begin{rem}(*)
Alternatively, one can replace $W(g)$ by Perelman's invariant $A(\alpha,g)$ introduced in~\cite{Per3}
and follow~\cite{Mor-Tia} for the derivation of the inequality analogous to~(\ref{eq:W}). Working
with weak solutions in our sense as opposed to Ricci flow with $\delta$-cutoff leads to a
similar simplification.
\end{rem}

\section{The aspherical case}
\label{sec:hyper}

In this section, we outline the proof of the other `half' of the geometrization
 theorem~\ref{geometrisation}, where $\pi_1M$ is assumed to be infinite.
Subsection~\ref{sub:graph} is devoted to background material on Haken
manifolds, graph manifolds and the JSJ-decomposition. In Subsection~\ref{sub:flow},
we state Theorem~\ref{existence 2}, which contains all the results on the Ricci flow that
are needed in the proof. In Subsection~\ref{sub:seq}, we explain how to use this result
to reach the desired conclusion.

\subsection{Some 3-manifold theory}\label{sub:graph}

 A \bydef{Haken} $3$-\var\ is a compact, \irr\ $3$-\var\ that contains an \incomp\ surface.
A $3$-\var\ is \bydef{atoroidal} if every incompressible torus in $M$ is parallel to a
component of $\bord M$.
Jaco-Shalen~\cite{js:seifert} and Johannson~\cite{joh:hom} showed that each closed Haken $3$-\var\ contains a canonical family
of disjoint  \incomp\ tori $T_1,\ldots,T_n$, called the \bydef{JSJ-decomposition} of $M$, such that each connected component of $M$ split along
$T_1,\ldots,T_n$ is Seifert fibered or atoroidal. Thurston proved that the atoroidal
pieces which are not Seifert are hyperbolic (see~\cite{otal:fibered,otal:haken,kap:hyp}.) Note that the JSJ-decomposition may be empty. In this case, $M$ is Seifert or
hyperbolic. Therefore, the above-mentioned results prove Theorem~\ref{geometrisation} in the case
where $M$ is Haken.

We say that $M$ is a \bydef{graph manifold} if it is a union of circle bundles glued
along their boundaries. This notion was introduced by F.~Waldhausen~\cite{wald:klasse}.
We collect some useful facts about graph manifolds in the
next proposition:
\begin{prop}\label{prop:graph}
\begin{enumerate}
\item Any Seifert manifold is a graph manifold.
\item If $M$ is a Haken manifold, then $M$ is a graph manifold if
and only if its JSJ-decomposition has only Seifert pieces.
\item If $M$ is an irreducible graph manifold, then $M$ is Seifert or contains an \incomp\ torus.
\end{enumerate}
\end{prop}

We shall be concerned with the long time behavior of the Ricci flow, or weak solutions
if Ricci flow is not defined for all time. More precisely, we look at the hyperbolically
rescaled solution $\tilde g(t)=(4t)^{-1} g(t)$. Heuristically, one can expect $\tilde g(t))$
to converge to a hyperbolic metric on $M$ if there exists one, and to collapse if $M$ is
a graph manifold. If $M$ has a nonempty JSJ-decomposition with at least one hyperbolic
piece, then it should split into a `thick' part and a `thin' part, where the thick part
corresponds to the union of the hyperbolic pieces, and the `thin' part corresponds to
the union of the Seifert pieces, which is a graph manifold. These two parts should be
separated by \incomp\ tori.

\begin{rem}
It follows from a theorem of R.~Myers~\cite{Myers}
that any closed $3$-\var\ $M$ contains a \emph{hyperbolic knot}, i.e.~an embedded circle whose
complement is hyperbolic. In particular, $M$ can be decomposed as the union of a hyperbolic
manifold and a solid torus (which is a graph manifold). Thus when we work with the
riemannian metrics given by weak solutions to the Ricci flow, it is crucial to show
that the tori that appear as boundary components of the thick part are \incomp; otherwise,
we would not obtain any further understanding of the topology of $M$.
\end{rem}
\medskip

Before closing this subsection, we define another classical notion in $3$-manifold
topology which is needed in the sequel:
let $X$ be a compact $3$-\var\ whose boundary is a (possibly empty) union of tori. We say
that $M$ \bydef{can be obtained from $X$ by Dehn filling} if there exists an embedding
$f:X\to M$ and a (possibly empty, possibly disconnected) $1$-submanifold $L\subset M$ such that $M\setminus f(X)$ is a regular
neighborhood of $L$. By convention, we allow $L$ to be empty, so that the class of manifolds
that can be obtained from $M$ by Dehn filling contains $M$ itself. This will be convenient later on.

If $Y$ is an open $3$-\var\ diffeomorphic to $X\setminus \bord X$ and $M$ can be
obtained from $X$ by Dehn filling, then we also say that $M$ can be obtained from $Y$ by
Dehn filling. Thus the theorem of Myers quoted above implies that any closed $3$-\var\ can
be obtained from a hyperbolic manifold by Dehn filling.

\subsection{Long time behavior of weak solutions}
\label{sub:flow}

Until the end of Section~\ref{sec:hyper}, $M$ is a closed, \irr\ $3$-\var\ whose fundamental
group is infinite. By Lemma~\ref{lem:classic}, this implies that  $M$ is aspherical.

Using the contrapositive of Theorem~\ref{reconnait topo}, it follows from our hypotheses that $M$
does not admit locally canonical metrics. Hence Theorem~\ref{existence 1} implies that for
any initial condition, there exists a weak solution defined on any compact interval
$[0,T]$. We need to strengthen this in two respects: first we need weak
solutions defined on $[0,+\infty)$; second, we want them to satisfy a list of geometric properties
relevant to the  topological applications. In order to state those properties,
we need some terminology.

\begin{defi}
Let $k>0$ be a whole number. Let $(M_n,g_n,x_n)$ be a sequence of pointed riemannian manifolds,
and let $(M_\infty,g_\infty,x_\infty)$ be a riemannian manifold. We shall say that $(M_n,g_n,x_n)$
\bydef{converges to} $(M_\infty,g_\infty,x_\infty)$ in the $\mathcal C^k$-sense if
there exists a sequence of numbers $\epsi_n>0$ tending to zero, and a sequence
of $\mathcal C^k$-diffeomorphisms $\varphi_n$ from the metric ball $B(x_\infty,\epsi_n^{-1})\subset M_\infty$
to the metric ball $B(x_n,\epsi_n^{-1})\subset M_n$ such that $\varphi_n^*(g_n)-g_\infty$
has $\mathcal C^k$-norm less then $\epsi_n$ everywhere. We say that the sequence
\bydef{subconverges} if it has a convergent subsequence.
\end{defi}

\begin{rem}
Note that the limit manifold $M_\infty$ need not be homeomorphic to any of the $M_n$'s.
Typically, the $M_n$'s are compact and $M_\infty$ is noncompact. However, if
$M_\infty$ is compact, then for large $n$, the manifold $M_n$ must be diffeomorphic to
$M_\infty$.
\end{rem}

\begin{defi}
Let $(X,g)$ be a riemannian $3$-\var. We say that a point $x\in X$ is $\epsi$-\bydef{thin
with respect to} $g$ if there exists a radius $\rho\in (0,1]$ such that the ball
$B(x,\rho)$ has the following two properties: all sectional curvatures on this ball are
bounded below by $\rho^{-2}$, and the volume of this ball is less than $\epsi \rho^3$.
Otherwise, $x$ is $\epsi$-\bydef{thick with respect to} $g$.
\end{defi}

\medskip
We set $$\hat R(g) := \Rmin(g) \cdot \vol(g)^{2/3}.$$
This quantity has two important properties: it is scale-invariant, and
it is nondecreasing along the Ricci flow on a closed manifold, as long as it remains
nonpositive. This is also true for weak solutions by condition~(ii) of Definition~\ref{def:ws}.
  Since our manifold $M$ is aspherical, it does not admit any metric of
  positive scalar curvature~\cite{gl:scalar,sy:scalar}. Hence $\hat R(g)$ is nonpositive
  for each metric $g$ on $M$. As a consequence, we have:
\begin{prop}\label{prop:rhat}
Let $\{g(t)\}$ be a weak solution on $M$. Then the function $t\mapsto \hat R(g(t))$ is nondecreasing.
\end{prop}

When $H$ is a hyperbolic manifold, we let $\hat R(H)$ denote $\hat R(g_\hyp)$,
 where $g_\hyp$ is the hyperbolic metric. Note that this number is equal to
 $-6\cdot  \vol(g_\hyp)^{2/3}$, since $g_\hyp$ has constant scalar curvature equal to $-6$.
 Hence if $H_1,H_2$ are two hyperbolic manifolds, then $\vol(H_1)\le \vol(H_2)$ if
 and only if $\hat R(H_1)\ge \hat R(H_2)$.

We are now in position to state the main result of this subsection (cf.~\cite[Sections 6 and 7]{Per2}):
\begin{theo}\label{existence 2}
Let $M$ be a closed, \irr, aspherical $3$-\var. For every riemannian metric $g_0$ on $M$,
there exists a weak solution $g(t)$ defined on $[0,+\infty)$ with
the following properties:
\begin{enumerate}
\item $g(0)=g_0$
\item The volume of the hyperbolically rescaled metric $\tilde g(t)$ is bounded uniformly in $t$.
\item For every $\epsi>0$ and every sequence $(x_n,t_n) \in M\times [0,+\infty)$, if $t_n$ tends to
$+\infty$ and $x_n$ is $\epsi$-thick
with respect to $\tilde g(t_n)$, then there exists a hyperbolic $3$-manifold $H$ and a basepoint $*\in H$
such that the sequence $(M,\tilde g(t_n),x_n)$
subconverges in the pointed $\mathcal C^2$ topology to $(H,g_\hyp,*)$. (Recall
that for us, `hyperbolic manifold' means complete and of bounded volume.)
\item For every sequence $t_n\to\infty$, the sequence $\tilde g(t_n)$ has controlled curvature in the sense of Perelman.
\end{enumerate}
\end{theo}

`Controlled curvature in the sense of Perelman' is a technical property, which is weaker
than a global two-sided curvature bound, but suffices for some limiting arguments.
It will not be discussed here. See~\cite{b3mp:hyper6} for the definition.

\begin{rem}
The assumption~that $M$ is irreducible is in fact redundant: since $M$ is aspherical,
any embedded $2$-sphere in $M$ bounds a homotopy $3$-ball $B$; by the positive
solution to the Poincar\'e Conjecture, $B$ must be diffeomorphic to the $3$-ball. We include this hypothesis to emphasize that various parts of the proof of
Theorem~\ref{geometrisation} can be made independent from one another.
\end{rem}

\begin{rem}
Condition~(iii) may be vacuous, i.e.~there may exist no such sequence of $\epsi$-thick basepoints
for any fixed $\epsi$. This happens for instance if $g(t)$ is a flat solution on a $3$-torus.
\end{rem}

\subsection{Sequences of riemannian metrics on aspherical $3$-manifolds}
\label{sub:seq}

We begin with a direct corollary of Theorem~\ref{existence 2}:
\begin{corol}\label{corol:existence 2}
Let $M$ be a closed, \irr, aspherical $3$-\var. For every riemannian metric $g_0$ on $M$,
there exists an infinite sequence of riemannian metrics $g_1,\ldots,g_n,\ldots$ with
the following properties:
\begin{enumerate}
\item The sequence $(\hat R(g_n))_{n\ge 0}$ is nondecreasing. In particular, it has a limit,
which is greater than or equal to $\hat R(g_0)$.
\item The sequence $(\vol(g_n))_{n\ge 0}$ is bounded.
\item For every $\epsi>0$ and every sequence $x_n\in M$, if $x_n$ is $\epsi$-thick
with respect to $g_n$, then there exists a hyperbolic $3$-manifold $H$ and a basepoint $*\in H$
such that the sequence $(M,g_n,x_n)$
subconverges in the pointed $\mathcal C^2$ topology to $(H,g_\hyp,*)$.
\item The sequence $g_n$ has controlled curvature in the sense of Perelman.
\end{enumerate}
\end{corol}

\begin{proof}
Applying Theorem~\ref{existence 2}, we obtain a weak solution $\{g(t)\}_{t\in [0,\infty)}$ with initial condition $g_0$.
Set $g_n:=\tilde g(t_n)$ for $n\ge 1$. Using Proposition~\ref{prop:rhat} and the
scale invariance of $\hat R$, we have $$\hat R(g_0)\le
\hat R(g_1) \le \cdots \le \hat R(g_n) \le \cdots.$$
This proves~assertion~(i). Assertions~(ii), (iii) and (iv)
follow from their counterparts in the conclusion of Theorem~\ref{existence 2}. 
\end{proof}

The next task is to explore the topological consequences of the existence of
a sequence of metrics satisfying the conclusion of Corollary~\ref{corol:existence 2}. This
leads to Proposition~\ref{summary} below.

\begin{prop}\label{summary}
Let $g_0,g_1,\dots$ be a sequence of riemannian metrics on $M$ as in Corollary~\ref{corol:existence 2}.
Then one of the following conclusions hold:
\begin{enumerate}
\item $M$ is a graph manifold,
\item $M$ is hyperbolic,
\item $M$ contains an incompressible torus, or
\item There exists an open hyperbolic manifold $H$ such that $M$ can be obtained
by Dehn filling on $H$, and $\hat R(H)\ge \hat R(g_0)$.
\end{enumerate}
\end{prop}

\begin{proof}[Sketch of proof]
Up to extracting a subsequence, we distinguish several cases:

\paragraph{Case 1 (`collapsing case')} There exists a sequence $\epsi_n\to 0$ such that every point
of $(M,g_n)$ is $\epsi_n$-thin. In this case, we prove that $M$ is a graph manifold. Below
we only give a quick sketch so that the reader can see the ideas involved.
See~\cite{b3mp:hyper6} for the details.

Our approach is centered around the notion of \bydef{simplicial volume}, introduced by
Gromov in~\cite{GromovIHES}. The simplicial volume of a closed, orientable $n$-manifold $X$ is defined by
$$\| X \| := \inf \{ \sum_i  |\alpha_i|, [X] = \sum_i \alpha_i\sigma_i\},$$
where $[X]\in H_n(X,\Rr)$ is the fundamental class, the $\sigma_i$'s are continuous maps
of the standard $n$-simplex to $X$, and the $\alpha_i$'s are real numbers.

It is known~\cite{Som} that if $X$ is a Haken $3$-manifold, then $\| X \|$ equals $V_3$ times
the sum of the volumes of the hyperbolic pieces in the JSJ-decomposition of $X$, where
$V_3$ is a universal constant. In
particular,  $\| X \|=0$ if and only if $X$ is a graph manifold.

Using the Cheeger-Gromov compactness theorem, we show that
every point $x\in M$ has a neighborhood $U_x$ whose geometry approximates a ball in a
$3$-\var\ of nonnegative curvature. By Cheeger-Gromoll theory, there is a list of possible
topologies for $U_x$; for instance, $U_x$ might be a thickened torus $T^2\times I$ or
a solid torus $S^1\times D^2$. All the $U_x$'s have virtually abelian fundamental group.
Using the hypothesis that $M$ is aspherical, we show that there exists $x\in M$ such that the complement $X$ of $U_x$ in $M$
is Haken. The next step is to prove that any closed $3$-\var\ obtained from $X$ by Dehn filling has
vanishing simplicial volume. Using Thurston's hyperbolic Dehn filling theorem and classical facts
from $3$-manifold topology, one deduces that $X$ is a graph manifold, which implies that
$M$ is a graph manifold. 

\paragraph{Case 2} There exists $\epsi>0$ and a sequence $x_n\in M$ such
that $x_n$ is $\epsi$-thick with respect to $g_n$. 

Using Corollary~\ref{corol:existence 2}(iii), by further extracting a subsequence we can
assume that $(M,g_n,x_n)$ converges in the
pointed $\mathcal C^2$-sense to some pointed hyperbolic manifold $(H^1,g_\hyp,*)$.
If $H^1$ is closed, then $M$ is
diffeomorphic to $H^1$, hence $M$ is hyperbolic. Thus the interesting case is
when $H^1$ is noncompact. In this case, $(M,g_n)$ contains for large $n$ a \svar\ $\bar H^1_n$
which is a large metric ball around $x_n$ and diffeomorphic to a large ball in $H$.
Thus each boundary component of $\bar H^1_n$ is a torus corresponding to some
cusp cross-section in $H$.

Repeating this construction if necessary, we find a finite set of hyperbolic manifolds $H^1,\ldots,
H^p$ such that, for large $n$, the $\epsi$-thick part of $(M,g_n)$ is covered by disjoint \svar s $\bar H^1_n,
\ldots,\bar H^p_n$, which are approximated by large metric balls in the $H^i$'s, and are
bounded by approximately cuspidal tori. To prove that the construction stops for an integer $p$
independent of $n$, we use the uniform bound on $\vol(g_n)$ and the Margulis Lemma.

All boundary components of the $\bar H^i_n$'s are tori. If one of these tori is \incomp\ in $M$, then
we are done. Hence we assume that they are all compressible. Let $X$ be a connected \svar\ of $M$.
We say that $X$ is \bydef{abelian}
if the image of the natural homomorphism $\pi_1 X\to \pi_1 M$ is abelian. This is the case,
for instance, if some component of $\bord X $ bounds a solid torus containing $\bar H^i_n$, or
if $\bar H^i_n$ is contained in some topological $3$-ball.

We have the following purely topological lemma:
\begin{lem}\label{lem:topo}
Let $X\subset M$ be a \svar\ bounded by compressible tori. If $X$ is non-abelian, then $M$
can be obtained from $X$ by Dehn filling.
\end{lem}

The proof then splits again into two cases: if all $\bar H^i_n$'s are abelian, then a refined version of
the argument used for Case~1 shows that $M$ is a graph manifold. If there exists $i$ such that
$\bar H^i_n$ is non-abelian, then using Lemma~\ref{lem:topo} applied to $X=\bar H^i_n$,
we deduce that $M$ can be obtained from $\bar H^i_n$ by Dehn filling. Since
$H^i$ is diffeomorphic to the interior of $\bar H^i_n$, $M$ is also obtained from
$H^i$ by Dehn filling.
From the monotonicity of $\hat R(g_n)$ and the fact that $H^i$ is a pointed limit of $(M,g_n)$, we get
$$\hat R(g_0) \le \lim_{n\to\infty} \hat R(g_n) \le \hat R(H^i).$$ This finishes our sketch of proof of
Proposition~\ref{summary}.
\end{proof}

We continue the proof of Theorem~\ref{geometrisation}. If conclusion~(i), (ii), or~(iii) of
Proposition~\ref{summary} holds, then by Proposition~\ref{prop:graph} the required topological conclusion has been reached. All that
remains to do is to explain how the initial metric $g_0$ can be chosen so that
conclusion~(iv) is excluded.

At this point it is convenient to recall the following well-kown facts from hyperbolic geometry:
\begin{theo}[see~\cite{gromov:bourbaki}]\label{thm:jorgensen}
The set of volumes of hyperbolic manifolds is well-ordered.
\end{theo}

\begin{prop}[\cite{anderson:scalar}]\label{cusp closing}
Let $H$ be an open  hyperbolic $3$-manifold and $M$ be a closed $3$-manifold obtained
from $H$ by Dehn filling. Then $M$ carries a riemannian metric $g_\epsi$ such that
$\hat R(g_\epsi)> \hat R(H)$.
\end{prop}

Consider the collection of all hyperbolic manifolds $H$ such that $M$ can be obtained
from $H$ by Dehn filling. By the theorem of Myers quoted at the end of Subsection~\ref{sub:graph},
this collection is never empty. Hence we can consider
the infimum of the volumes of these manifolds, which we will denote by $V_0(M)$.
By Theorem~\ref{thm:jorgensen}, this infimum is in fact a minimum.\footnote{Although it is not necessary for the proof, it is perhaps worth remarking that when $M$ is
hyperbolic, one can show that $V_0(M)$ is equal to the hyperbolic volume of $M$.
This follows from the well-known fact that hyperbolic Dehn filling decreases volume,
which is also conceptually connected to Proposition~\ref{cusp closing}. Hence
one can think of $V_0(M)$ as a replacement for the hyperbolic volume when the manifold
is not hyperbolic.}

We are now ready for the last argument: let $H_0$ be a hyperbolic manifold realizing the
minimum in $V_0(M)$. If $M$ is not hyperbolic, then $H_0$ is open, and $M$ is
obtained from $H_0$ by Dehn filling. By Proposition~\ref{cusp closing}, $M$ admits
a metric $g_\epsi$ such that $\hat R(g_\epsi)> \hat R(H_0)$. Applying Theorem~\ref{existence 2}
with $g_0=g_\epsi$ yields a sequence $g_1,g_2,\ldots$. If $H$ is a hyperbolic manifold
from which $M$ can be obtained by Dehn filling, then by definition of $V_0(M)$, we have
$\vol(H)\ge \vol(H_0)$. This implies that $\hat R(H)\le \hat  R(H_0)<\hat R(g_\epsi)$. Thus
conclusion~(iv) of Proposition~\ref{summary} is excluded, and applying this proposition
finishes the proof of Theorem~\ref{geometrisation} in the aspherical case.

\begin{rem}
Since we allow ourselves to pass to subsequences, we do not prove that when $M$ is hyperbolic,
hyperbolically rescaled weak
solutions starting from arbitrary metrics actually converge to the hyperbolic metric. This stronger
statement is in fact true. Its proof requires additional arguments
(see~\cite{Per2,hamilton:nonsingular,Kle-Lot}.)
\end{rem}

\begin{rem}(*)
The part of the proof that deals with compressible tori is inspired by~\cite[Section 8]{Per2}.
We have replaced Perelman's invariants $\hat\lambda$
and $\bar V$ by $\hat R$ and $V_0$ respectively. The idea to use $\hat R$ seems due to M.~Anderson
(see also~\cite[Section 93]{Kle-Lot}, where two versions of the argument along the lines suggested by Perelman
are given.)  The minimal volume $\bar V$ considered by Perelman and Kleiner-Lott is different from our $V_0$; for
instance, $\bar V(M)$ is zero if $M$ is a graph manifold, whereas $V_0(M)$ is always positive.

Our treatment of the collapsing case is completely different from the one hinted at by
Perelman~in~\cite[Section 7]{Per2}. See~\cite{sy} for another approach using Alexandrov
space theory.
\end{rem}

\section{More on weak solutions}
\label{sec:solution}
The purpose of this section is to discuss some aspects of the proof of Theorem~\ref{existence 1}. We first
give some background on the Ricci flow, then describe the metric surgery construction.
In all of this section, $M$ is a closed, \irr, $RP^2$-free $3$-manifold. 

For  an evolving metric $\{g(t)\}$ we denote the minimum (resp.~maximum) of the scalar curvature of the 
time-$t$ metric by $R_\mathrm{min} (t)$ (resp.~$R_\mathrm{max}(t)$.)

\subsection{Preliminaries}

A riemannian metric on $M$ is \bydef{normalized} if it has sectional curvature between
$1$ and $-1$, and the volume of any ball of unit  radius is greater than or equal to half the
volume of a Euclidean ball of unit radius. Since $M$ is compact, any metric can be normalized
by scaling. Moreover, the property of being locally canonical for a metric is scale invariant.
Hence it suffices to prove Theorem~\ref{existence 1} for normalized initial conditions.

Let  $I\subset [0,+\infty)$ be an interval. Following the terminology of~\cite{Mor-Tia}, we say
that an evolving metric $\{g(t)\}_{t\in I}$ has \bydef{curvature pinched toward positive}
if for every $(x,t)\in M\times I$  the following two conditions hold: 
\begin{gather}
R(x,t)  \ge - \frac{6}{4t+1} \\
R(x,t) \ge 2X(x,t) (\log X(x,t) + \log(1 + t) - 3) \qquad \text{whenever\ } X(x,t)>0,
\end{gather}
where $X(x,t)$ is the opposite of  the lowest eigenvalue of the curvature
operator of $g(t)$ at $x$.

Below we give a few basic properties of the Ricci flow on $M$.

\begin{prop}[Long time existence]\label{long time}
Let $K$ be a positive number. If $g_0$ is a metric satisfying $\vert \Rm \vert\le K$, then the Ricci flow
solution with initial condition $g_0$ exists on $[0,2^{-4} K^{-1}]$ and satisfies
$\vert \Rm \vert\le 2K$ on this interval. 
\end{prop}

\begin{prop}[Curvature estimates]\label{curvature}
Let $I\subset [0,+\infty)$ be an interval containing $0$. Let $\{g(t)\}_{t\in I}$ be a Ricci flow
solution with normalized initial condition. Then $g(t)$ has curvature pinched toward positive.
\end{prop}

Seeking to prove Theorem~\ref{existence 1}, we fix a number $T>0$ and a normalized initial condition $g_0$. It follows from Proposition~\ref{curvature} that $R_\mathrm{min}(t)$ is uniformly bounded from below.
Moreover, if we have a solution defined on some interval $[0,t_0]$ and such that
$R_\mathrm{max}(t)$ is bounded from above on $[0,t_0]$, then by
Proposition~\ref{curvature}, the norm of the curvature tensor is also bounded above on this interval.
Hence by Proposition~\ref{long time}, the solution can be prolonged on a slightly larger interval
$[0,t_0+\alpha]$.

As a result, Theorem~\ref{existence 1} is only difficult to prove if the Ricci flow solution
with initial condition $g_0$ is defined on a maximal interval $[0,T')$ with $T'<T$ and
$R_\mathrm{max}(t)$ is unbounded as $t\to T'$. That is why we need a good description of the regions of $M$ where the scalar curvature becomes large. Such a description is provided by the following
theorem of Perelman (cf.~\cite[Theorem 12.1]{Per1}.)

\begin{theo}\label{WC}
For every $\epsi>0$ and every $T>0$, there exists $r=r(\epsi,T)>0$ such that if $\{g(t)\}$ is a Ricci flow
solution on $M$ defined on $[0,T]$ with normalized initial condition, then for all $(x_0,t_0)\in M\times
[0,T]$ such that $R(x_0,t_0)\ge r^{-2}$, the point $(x_0,t_0)$ has an $\epsi$-canonical neighborhood
in $(M,\{g(t)\})$.
\end{theo}

The definition of $\epsi$-canonical neighborhood is a bit technical, so it will not be given here.
It suffices to know two things: first, if $(x_0,t_0)$ has an $\epsi$-canonical neighborhood
in $(M,\{g(t)\})$, then $x$ has a weak  $\epsi$-canonical neighborhood in $(M,g(t_0))$;
second, it implies the differential inequality
\begin{equation}\label{eq:dRdt}
\frac{\partial R}{\partial t} < C \vert R \vert^2,
\end{equation}
where $C$ is a constant depending only on $\epsi$. (In the sequel, $\epsi$ will be fixed,
so $C$ will be universal.)

\begin{defi}
An evolving metric $\{g(t)\}_{t\in I}$ satisfies the \bydef{Canonical Neighborhoods} property
with parameter $r$ (henceforth abbreviated as $(CN)_r$) if  for all $(x,t)\in M\times
[0,T]$ such $R(x,t)\ge r^{-2}$, the point $(x,t)$ has an $\epsi$-canonical neighborhood
in $(M,g(t))$.
\end{defi}

\subsection{Surgery on $\delta$-necks}
We now describe the surgery procedure. We fix a small number $\epsi>0$. This number
should satisfy various conditions, e.g.~it should be less than ${10}^{-2}$ so that
Theorem~\ref{reconnait topo} holds.

The two main parameters that govern the surgery procedure are called $r$ and $\delta$.  The
parameter $r$ is related to the curvature scale above which points have canonical
neighborhoods; it has the dimension of length, and is smaller than the number $r(\epsi,T)$
given by Theorem~\ref{WC}. The parameter $\delta$ describes the
precision of the surgery; it is dimensionless, and much smaller than $\epsi$.
Assume for the moment that $r$ and $\delta$ have been fixed.

The following technical result is adapted from Lemma 4.3 of~\cite{Per2}.
\begin{theo}[Existence of cutoff parameters]\label{cutoff}
There exist positive numbers $h<\delta r$ and $D\ge 1$ depending only on $\delta$ and $r$, such that if
$(M,\{g(t))\}$ is a weak solution satisfying $(CN)_r$, $t$ is a time in the domain of definition,
and $x,y,z$ are
points of $M$ such that $R(x,t)\le 2r^{-2}$, $R(y,t)=h^{-2}$, $R(z,t)\ge Dh^{-2}$,
and $y$ lies on a minimizing $g(t)$-geodesic connecting $x$ to $z$, then $y$ is the center
of a $\delta$-neck.
\end{theo}

We now carry out the construction outlined in Section~\ref{sec:poincare}, setting the
curvature threshold to $2Dh^{-2}$. Precisely, this means the following: we let $t_0\le T$
be the first time where $R_\mathrm{max}(t)$ reaches $\Theta:=2Dh^{-2}$ (if there is no such time, then
there is nothing to prove.) If all points of $(M, g(t_0))$
have scalar curvature greater than $r^{-2}$, then by Theorem~\ref{WC} they all have canonical neighborhoods, and we are done. From now on, we assume that it is not the case.
We partition $M$ into three subsets $\mathcal R,\mathcal O,\mathcal G$ defined as follows:

\begin{align*}
 \mathcal R &:= \{z \in M \mid Dh^{-2} \le R(z,t_0)\}\\
\mathcal O &:= \{ y\in M \mid 2 r^{-2} < R(y,t_0) < Dh^{-2}\} \\
\mathcal G &:= \{ x\in M \mid R(x,t_0) \le 2r^{-2} \} .
\end{align*}

Intuitively, $\mathcal R$ is the set of `red' points, where the scalar curvature is huge, and
a singularity is threatening to appear. The purpose of the surgery operation is to
remove those points. The set $\mathcal G$ consists of `green' points, where
$R$ cannot be large, and $\mathcal O$ is the set of `orange' points, where $R$ may be
large, but not that much.

By assumption, $\mathcal G$ and $\mathcal R$ are not empty. Let $x$ be a
point of $\mathcal G$, $z$ be a point of $\mathcal R$, and $\gamma$ be a minimizing
geodesic connecting $x$ to $z$. Then by the intermediate
value theorem, there exists a point $y\in\gamma$ whose scalar curvature is exactly $h^{-2}$.
Applying Theorem~\ref{cutoff}, we obtain a $\delta$-neck centered on $y$.
One can show that this can be repeated finitely many times to yield a finite
collection of disjoint $\delta$-necks $N_1,\ldots N_p$ whose union separates $\mathcal R$ from $\mathcal G$.

Let us denote by $X_1,\ldots,X_q,X_{q+1},\ldots, X_{q+s}$ the connected components of
$M\setminus \bigcup_i N_i$, where $X_j\subset \mathcal G\cup \mathcal O$ for $j\le q$
and $X_j\subset\mathcal O\cup\mathcal R$ for $j>q$.
By a relative version of Theorem~\ref{reconnait topo}, the high curvature components
$X_{q+1},\ldots, X_{q+s}$, which are covered by canonical neighborhoods, 
are diffeomorphic to $B^3$ or $S^2\times I$. Since $M$ is irreducible, a straightforward topological argument shows that
there is a not-so-large-curvature component $X_{j_0}$, $j_0\le q$, such that
$M\setminus X_{j_0}$ is a finite, disjoint union of topological
$3$-balls in $M$, and each point of $\bord X_{j_0}$ is the center of a $\delta$-neck.

In other words, $M\setminus X_{j_0}$ is covered by a union $\Sigma(t_0)$ of $\delta$-caps. The surgery
operation consists in replacing those caps by special caps, called `almost standard
caps'. The precise definition is too technical to be discussed here. The main points
are that the post-surgery metric $g_+(t_0)$ has curvature pinched toward positive, and
at all points of $\Sigma(t_0)$ its scalar curvature is comparable to $h^{-2}$. As a result,
$R_\mathrm{max} (g_+(t_0))$ is bounded above by $\Theta/2$.
After this, we restart the Ricci flow with new initial condition $g_+(t_0)$.

As long as condition $(VC)_r$ is satisfied, we can iterate this construction. Since points
of scalar curvature between $r^{-2}$ and $\Theta$ have canonical neighborhoods, they
satisfy inequality~(\ref{eq:dRdt}). Together with the fact that surgery makes $\Rmax$ drop
by at least half its value, this ensures that there is a definite lower bound for the time span
between two consecutive surgeries. Hence the iteration of this process produces a weak
solution.

To prove Theorem~\ref{existence 1}, we need to show that the parameters $r$ and $\delta$
can be chosen so that the construction can indeed be iterated until it produces a
weak solution defined on $[0,T]$ or a locally canonical metric.
Theorem~\ref{WC} does not
suffice for this; instead, we need to generalize it to a special class of weak solutions.
This is done in a complicated limiting argument involving such notions as $\kappa$-solutions
or $\kappa$-noncollapsing, which I will not attempt to explain here.

\begin{rem}(*)
Since we work on a compact time interval rather than $[ 0,+\infty)$, the parameters $r$ and $\delta$ are fixed rather than time-dependent as in~\cite{Per2}. This simplification was
observed by Perelman in~\cite{Per3}. To prove Theorem~\ref{existence 2} however, we need
to work with time-dependent parameters, which creates an additional layer of complexity.
\end{rem}

\begin{rem}(*)
As we already noticed, the main difference between our construction and Perelman's is that we do surgery before the singularity appears, rather than at the singular time. As a consequence, we do not
have to discuss horns, capped horns and double horns. Our solution
for a given initial condition may be very different from Perelman's: at one extreme,
Ricci flow may be defined for all time: then Perelman's construction produces the Ricci flow
solution, while ours may lead to surgery if the curvature reaches the threshold. Neither
construction leads to a \emph{canonical} `Ricci flow through singularities'.

Another consequence of this choice is that we do not explicitly address the question of formation
of singularities in the Ricci flow. The reader should be aware, however, that the difficulties we
have to face and the way they are overcome (e.g.~blow-up arguments) are essentially the same.
\end{rem}

\begin{rem}(*)
Our argument to rule out accumulation of surgeries is different from Perelman's, and does not
use the volume. This is crucial for the generalization to infinite volume solutions discussed
below.
\end{rem}

\section{Extension to non-compact 3-manifolds}

Let $M$ be a possibly noncompact $3$-manifold without boundary. A riemannian metric
$g$ on $M$ has \bydef{bounded geometry} if it has bounded sectional curvature
and injectivity radius bounded away from zero. An evolving metric $\{g(t)\}_{t\in I}$ is
\bydef{complete} (resp.~\bydef{has bounded geometry}, resp.~\bydef{has bounded
sectional curvature}) if for each $t\in I$ the metric
$g(t)$ has the corresponding property.

The results of Section~\ref{sec:poincare}
can be extended to this context. For simplicity we give the existence result for \emph{irreducible}
open $3$-manifolds (note that such a manifold is automatically $RP^2$-free.)

\begin{theo}\label{existence open}
Let $M$ be an open, \irr\  $3$-manifold not diffeomorphic to $\Rr^3$. Let $g_0$ be a complete riemannian metric on $M$
which has bounded geometry. Then there exists a complete
weak solution $\{g(t)\}_{t\in [0,\infty)}$ of bounded geometry
on $M$ with initial condition $g(0)=g_0$.
\end{theo}

Theorem~\ref{existence open} is proved using the construction described in Section~\ref{sec:solution}
and the following version of Theorem~\ref{reconnait topo} for
open $3$-manifolds:

\begin{theo}\label{reconnait topo open}
Let $(M,g)$ be an open riemannian $3$-manifold. If $g$ is locally canonical,
then $M$ is diffeomorphic to $\Rr^3$ or $S^2\times \Rr$. In particular, if $M$ is \irr, then $M$ is diffeomorphic to $\Rr^3$.
\end{theo}

\begin{proof}[Sketch of proof of Theorem~\ref{existence open}]
One defines parameters $r,\delta,h,D,\Theta$ as in the compact case.
If the maximum of the scalar curvature reaches the threshold $\Theta$ for some time $t_0$, there is a
similar decomposition of $M$ into three subsets $\mathcal R$, $\mathcal O$, and $\mathcal G$,
which may have infinitely many components. By Theorem~\ref{reconnait topo open}, $g(t_0)$ cannot be locally
canonical, so $\mathcal G$ is nonempty. Using an generalization of Theorem~\ref{cutoff}, one finds a (possibly infinite) set of
spheres in $\mathcal O$ which are middle spheres of $\delta$-necks and separate $\mathcal
R$ from $\mathcal G$.  Using elementary $3$-manifold topology, one deduces from
irreducibility of $M$ that there is a noncompact submanifold $X\subset M$ containing
no point of $\mathcal R$, and whose complement is a (possibly infinite) set of disjoint
$3$-balls. Then one performs geometric surgery on each of these $3$-balls in order to
decrease the maximal scalar curvature. The rest of the argument goes through without
major changes.
\end{proof}

Theorem~\ref{existence open} has an application to positive scalar curvature, which we now describe. Let
us denote by $\Rmin(g)$ the infimum of the scalar curvature of a metric $g$ (which
may or may not be attained.) Following~\cite{gl:scalar}, we say that $g$ has
\bydef{uniformly positive scalar curvature} if $\Rmin(g)>0$.

\begin{theo}\label{positive scalar irreducible}
Let $M$ be an open, \irr\ $3$-manifold which carries a complete metric $g_0$ of bounded
geometry and uniformly positive scalar curvature. Then $M$ is diffeomorphic to $\Rr^3$.
\end{theo}

\begin{proof}
If $M$ is not diffeomorphic to $\Rr^3$, then we can apply Theorem~\ref{existence open}
to get a complete weak solution $\{g(t)\}$ of bounded geometry with initial condition $g_0$,
defined on $[0,+\infty)$.
A maximum principle argument (cf.~\cite[\S 2.1]{Cao-Zhu}) together with condition~(ii)(a) of the definition
of weak solution shows that
$$\Rmin(t) \ge \frac{\Rmin(0)}{1-2t \Rmin(0)/3}$$
if $g(t)$ is a complete weak solution with bounded sectional curvature.
This gives an upper bound on the lifetime of such a solution, hence a contradiction.
\end{proof}

Theorem~\ref{positive scalar irreducible} is a special case of a more general result. In order to state it, we need a definition
of connected sum allowing infinitely many summands.

If $\mathcal X$ is a class of closed $3$-manifolds, we say that a $3$-manifold $M$ is a
\bydef{connected sum} of members of $\mathcal X$ if there exists a locally finite
simplicial tree $T$ and a map $v\mapsto X_v$ which associates to each vertex of $T$
a manifold in $\mathcal X$, such that by removing from each $X_v$ as many $3$-balls
as vertices incident to $v$ and gluing the thus punctured $X_v$'s to each other along the
edges of $T$, one obtains a $3$-manifold diffeomorphic to $M$.

For instance, both $\Rr^3$ and $S^2\times\Rr$ can be viewed as connected sums of $S^3$'s,
the tree being a half-line in the former case, and a line in the latter.

It is easy to see that if one takes $\mathcal X$ to be the class containing spherical
$3$-manifolds and manifolds diffeomorphic to $S^2\times S^1$, then any connected sum
of members of $\mathcal X$ admits a complete riemannian metric of uniformly positive scalar curvature. The
following result is a partial converse to this:
\begin{theo}\label{positive scalar general}
Let $M$ be a $3$-manifold which carries a complete metric $g_0$ of bounded
geometry and uniformly positive scalar curvature. Then $M$ is a connected sum of
spherical manifolds and copies of $S^2\times S^1$.
\end{theo}

To prove Theorem~\ref{positive scalar general}, one needs to work with a more general notion of weak solution,
where the manifold is allowed to change with time. The surgery process may
disconnect the manifold, and break it into possibly infinite connected sums. However, the
essence of the proof is already contained in Theorem~\ref{existence open}.


\bibliographystyle{alpha}
\bibliography{ricci}

\end{document}